\documentclass[12pt,reqno]{amsart}

\usepackage{amssymb}
\usepackage[hypertex]{hyperref}

\begin{document}

\let\kappa=\varkappa
\let\epsilon=\varepsilon
\let\phi=\varphi
\let\p\partial

\def\Z{\mathbb Z}
\def\R{\mathbb R}
\def\C{\mathbb C}
\def\Q{\mathbb Q}
\def\P{\mathbb P}
\def\HH{\mathrm{H}}

\def\conj{\overline}
\def\Beta{\mathrm{B}}
\def\const{\mathrm{const}}
\def\ov{\overline}
\def\wt{\widetilde}
\def\wh{\widehat}

\renewcommand{\Im}{\mathop{\mathrm{Im}}\nolimits}
\renewcommand{\Re}{\mathop{\mathrm{Re}}\nolimits}
\newcommand{\codim}{\mathop{\mathrm{codim}}\nolimits}
\newcommand{\id}{\mathop{\mathrm{id}}\nolimits}
\newcommand{\Aut}{\mathop{\mathrm{Aut}}\nolimits}
\newcommand{\lk}{\mathop{\mathrm{lk}}\nolimits}
\newcommand{\sign}{\mathop{\mathrm{sign}}\nolimits}
\newcommand{\rk}{\mathop{\mathrm{rk}}\nolimits}
\def\Jet{{\mathcal J}}
\def\FC{{\mathrm{FCrit}}}
\def\sS{{\mathcal S}}
\def\lcan{\lambda_{\mathrm{can}}}
\def\ocan{\omega_{\mathrm{can}}}

\renewcommand{\mod}{\mathrel{\mathrm{mod}}}

\newtheorem{mainthm}{Theorem}
\renewcommand{\themainthm}{{\Alph{mainthm}}}
\newtheorem{thm}{Theorem}[subsection]
\newtheorem{lem}[thm]{Lemma}
\newtheorem{prop}[thm]{Proposition}
\newtheorem{cor}[thm]{Corollary}

\theoremstyle{definition}
\newtheorem{exm}[thm]{Example}
\newtheorem{rem}[thm]{Remark}
\newtheorem{df}[thm]{Definition}

\renewcommand{\thesubsection}{\arabic{subsection}}
\numberwithin{equation}{subsection}

\title{Cosmic censorship of smooth structures}
\author[Chernov \& Nemirovski]{Vladimir Chernov and Stefan Nemirovski}
\thanks{This work was partially supported by a grant from the Simons Foundation (\# 235674 to Vladimir Chernov).
The second author was supported by grants from DFG and RFBR}
\address{Department of Mathematics, 6188 Kemeny Hall,
Dartmouth College, Hanover, NH 03755, USA}
\email{Vladimir.Chernov@dartmouth.edu}
\address{%
Steklov Mathematical Institute, 119991 Moscow, Russia;\hfill\break
\strut\hspace{8 true pt} Mathematisches Institut, Ruhr-Universit\"at Bochum, 44780 Bochum, Germany}
\email{stefan@mi.ras.ru}

\begin{abstract}
It is observed that on many $4$-manifolds there is a unique
smooth structure underlying a globally hyperbolic Lorentz metric.
For instance, every contractible smooth $4$-manifold
admitting a globally hyperbolic Lorentz metric is diffeomorphic
to the standard~$\R^4$. Similarly, a smooth $4$-manifold
homeomorphic to the product of a closed oriented $3$-manifold $N$ and~$\R$
and admitting a globally hyperbolic Lorentz metric
is in fact diffeomorphic to $N\times \R$.
Thus one may speak of a censorship imposed by the global
hyperbolicty assumption on the possible smooth structures
on $(3+1)$-dimensional spacetimes.
\end{abstract}

\maketitle

\subsection*{Introduction}
One form of the {\it strong cosmic censorship hypothesis\/}
proposed by Penrose asserts that `physically relevant' spacetimes
should be globally hyperbolic (see~\cite{Penrose}).
The purpose of this note is to point out that global hyperbolicity
imposes strong restrictions on the {\it differential topology\/}
of the spacetime.
The starting point of all our considerations will be
the smooth splitting theorem for globally hyperbolic spacetimes
established by Bernal and S\'anchez~\cite{BernalSanchez, BernalSanchezMetricSplitting}.
All manifolds will be assumed Hausdorff and paracompact,
since Hausdorff spacetimes are necessarily paracompact by~\cite[pp.\,1743--1744]{Geroch}.

The first result is valid in all dimensions but seems to be
particularly interesting for $(3+1)$-dimensional spacetimes.
In that case, the argument makes essential use of the
three-dimensional Poincar\'e conjecture proved
by Perelman~\cite{Perelman1,Perelman2,Perelman3}.

\begin{mainthm}
\label{main}
Let $(X,g)$ be a globally hyperbolic $(n+1)$-dimensional spacetime.
Suppose that $X$ is contractible.
Then $X$ is diffeomorphic to the standard $\R^{n+1}$.
\end{mainthm}

For every $n\ge 3$, there exist uncountably many contractible
smooth $n$-manifolds that are not homeomorphic to $\R^n$
(see \cite{McMillan2}, \cite{Glaser} and~\cite{CurtisKwun}).
In dimension four, in addition to that there are uncountably
many smooth four-manifolds that are homeomorphic
but not diffeomorphic to $\R^4$
(the so-called {\it exotic\/} $\R^4$'s, see~\cite{Gompf} and~\cite{Taubes}).
The theorem shows that none of those carry globally
hyperbolic Lorentz metrics.

The topological argument used to prove Theorem~\ref{main}
in the $(3+1)$-dimensional case was first applied in the context
of Lorentz geometry by Newman and Clarke~\cite{NewmanClarke}.
They showed that a globally hyperbolic spacetime which {\it is\/}
diffeomorphic to $\R^4$ can have any contractible $3$-manifold
as its Cauchy surface, see Remark~\ref{NonStandardCauchy}.

Global hyperbolicity singles out `standard' smooth structures
on another large class of $(3+1)$-dimensional spacetimes as well.
The following result is based on Perelman's geometrization theorem for
$3$-manifolds and the work of Turaev~\cite{TuraevLong}.

\begin{mainthm}
\label{mainB}
Let $(X,g)$ be a globally hyperbolic $(3+1)$-dimensional spacetime.
Suppose that $X$ is homeomorphic to the product of a closed oriented
$3$-manifold $N$ and~$\R$.
Then $X$ is diffeomorphic to $N\times \R$, where $N$ and $\R$
have their unique smooth structures.
\end{mainthm}

In fact, we do not know an example of a topological $4$-manifold
admitting two non-diffeomorphic smooth structures each of which
is the smooth structure of a globally hyperbolic spacetime.
However, such manifolds exist in higher dimensions
(for instance, $S^7\times\R$). To show that $4$-dimensional
examples do not exist, one would need to prove Theorem~\ref{mainB}
for a $3$-manifold~$N$ that may be non-compact or
non-orientable.

\subsection{Globally hyperbolic spacetimes}
A {\it spacetime\/} is a time-oriented connected Lorentz manifold~$(X,g)$.
The Lorentz metric $g$ and the time-orientation define a distribution
of future hemicones in $TX$. A piecewise-smooth curve in $X$ is called
future-pointing if its tangent vectors lie in the future hemicones.

For two points $x,y\in X$, we write $x\le y$ if either $x=y$ or there exists
a future-pointing curve connecting $x$ to $y$. A spacetime
is called {\it causal\/} if $\le$ defines a partial order on it,
that is, if there are no closed non-trivial future-pointing
curves.

A spacetime $(X,g)$ is {\it globally hyperbolic\/} if it is
causal and the `causal segments' $I_{x,y}=\{z\in X\mid x\le z\le y\}$
are compact for all $x,y\in X$. (This definition is equivalent
to the classical one~\cite[\S 6.6]{HawkingEllis}
by~\cite[Theorem 3.2]{BernalSanchezCausal}.)

A {\it Cauchy surface\/} in a spacetime is a subset such that every
endless future-pointing curve meets it exactly once.
It is a classical fact~\cite[pp.\,211--212]{HawkingEllis} that
a spacetime is globally hyperbolic if and only if it contains a Cauchy
surface. It has long been conjectured (and sometimes tacitly assumed)
that Cauchy surfaces can be chosen to be smooth and spacelike
and that a globally hyperbolic spacetime must be diffeomorphic
to the product of its Cauchy surface with $\R$;
this was finally proved by Bernal and S\'anchez in 2003.

\begin{thm}[Bernal--S\'anchez \cite{BernalSanchez, BernalSanchezMetricSplitting}]
\label{Splitting}
For a globally hyperbolic $(n+1)$-dimensional spacetime $(X,g)$,
there exist an $n$-dimensional smooth manifold $M$ and a diffeomorphism
$h:M\times \R\to X$ such that
\begin{itemize}
\item[a)] $h(M\times \{t\})$ is a smooth spacelike Cauchy surface for all $t\in\R${\rm ;}
\item[b)] $h(\{x\}\times \R)$ is a future-pointing timelike curve for all $x\in M$.
\end{itemize}
\end{thm}

Note that it follows by projecting along the timelike $t$-direction
that {\it all\/} smooth spacelike Cauchy surfaces in $(X,g)$ are diffeomorphic
to the same manifold~$M$.

\subsection{Proof of Theorem~\ref{main}}
\label{ProofMain}
Suppose that $(X,g)$ is globally hyperbolic and $X$ is contractible.
By Theorem~\ref{Splitting},
we know that $X$ is diffeomorphic to the product $M\times\R$
for a smooth $n$-manifold $M$. Since $X$ is contractible, it follows
that $M$ is also contractible (as it is homotopy equivalent to $X$).
Thus, it remains to invoke the following result.

\begin{prop}[McMillan~\cite{McMillan,McMillanZeeman}, Stallings~\cite{Stallings}]
\label{Stab}
Suppose that $M$ is a contractible smooth $n$-manifold.
Then $M\times\R$ is diffeomorphic to $\R^{n+1}$.
\end{prop}

\begin{proof}
The proof splits into three cases according to the dimension of $M$.

\smallskip
\noindent
1. $\dim M\le 2$.

\smallskip
\noindent
The result is obvious because the only contractible manifolds of dimension $\le 2$ are $\R$ and $\R^2$.

\smallskip
\noindent
2. $\dim M=3$ (cf.~{\cite[p.\,55]{NewmanClarke}}).

\smallskip
\noindent
We outline McMillan's argument~\cite{McMillan,McMillanZeeman}
trying to give precise references for each step. For an introduction
to the relevant topological methods, the reader may consult the book
by Rourke and Sanderson~\cite{RourkeSanderson}.
McMillan \cite[Theorem 1]{McMillan} proved
that {\bf if the three-dimensional Poincar\'e conjecture holds true},
then $M$ can be exhausted by compact subsets PL-homeo\-morphic
to handlebodies with handles of index one.
It follows by an engulfing argument~\cite[Proof of Theorem 2, p.\,513]{McMillan}
that $M\times\R$ is the union of compact subsets
$B_n\subset M\times\R$ such that $B_n\subset \mathop{\mathrm{Int}} B_{n+1}$
and each $B_n$ is $\mathrm{PL}$-homeomorphic to the $4$-ball.
McMillan and Zeeman observed~\cite[Lemma~4]{McMillanZeeman}
that this implies that $M\times\R$ is $\mathrm{PL}$-homeomorphic to~$\R^4$.
However, if a smooth manifold is $\mathrm{PL}$-homeomorphic to $\R^n$,
then it is diffeomorphic to $\R^n$ by a result of Munkres \cite[Corollary 6.6]{Munkres}.
Since the Poincar\'e conjecture is now known to be true
because of Perelman's work~\cite{Perelman1,Perelman2,Perelman3},
the result follows.

\smallskip
\noindent
3. $\dim M\ge 4$.

\smallskip
\noindent
This is a special case of a result of Stallings~\cite[Corollary~5.3]{Stallings}.
\end{proof}

\begin{rem}[The r\^{o}le of the Poincar\'e conjecture]
The three dimensional Poincar\'e conjecture enters the preceding
argument in the case $n=3$ through the proof of \cite[Theorem 1]{McMillan}.
It is used there in the form of the following statement: {\it A null-homotopic
embedded\/ $2$-sphere in a three-manifold bounds a $3$-ball.}
The assertion that such a sphere bounds a {\it homotopy\/} ball
is classical and `elementary' (see e.\,g.~\cite[Proposition~3.10]{Hatcher});
the Poincar\'e conjecture ensures that the only homotopy $3$-ball
is the usual one.
\end{rem}

\begin{rem}[Standard spacetimes vs non-standard Cauchy surfaces]
\label{NonStandardCauchy}
Following Newman and Clarke~\cite{NewmanClarke}, let us show that
although the underlying manifolds of contractible globally hyperbolic
spacetimes are standard, their Cauchy surfaces can be completely
arbitrary:
{\it For every contractible smooth $n$-manifold $M$, there exists a globally
hyperbolic spacetime of the form $(\R^{n+1},g)$ with Cauchy surface
diffeomorphic to $M$.} Indeed, take any complete Riemann metric $\bar g$
on $M$, then $(M\times\R, \bar g\oplus -dt^2)$ is a globally hyperbolic
spacetime. By Proposition~\ref{Stab} the manifold $M\times\R$
is diffeomorphic to $\R^{n+1}$.
\end{rem}

\subsection{Proof of Theorem~\ref{mainB}}
The manifold $X$ is diffeomorphic to $M\times\R$ for some $3$-manifold $M$
by Theorem~\ref{Splitting}. We shall prove that $M$ is in fact homeomorphic
to $N$. Since homeomorphic $3$-manifolds are diffeomorphic~\cite[Theorem~3.6]{Munkres},
it will follow that the smooth $4$-manifolds $X\overset{\mathrm{diff}}{=}M\times\R$
and $N\times\R$ are diffeomorphic.

Note first that
$$
H_3(M,\Z)=H_3(M\times\R,\Z)=H_3(X,\Z)=H_3(N\times\R,\Z)=H_3(N,\Z)=\Z
$$
and hence $M$ is closed and orientable.

Turaev~\cite[Theorem 1.4, p.\,293]{TuraevLong} proved that two orientable
closed {\bf geometric} $3$-manifolds are homeomorphic if and only
if they are topologically $h$-cobordant. In~\cite[p.\,291]{TuraevLong}
geometric $3$-manifolds were defined as connected sums of
Seifert fibred, hyperbolic, and Haken manifolds.
It is now known by the work of Perelman~\cite{Perelman1, Perelman2}
that a non-Haken (and hence atoroidal) irreducible orientable
closed $3$-manifold is either Seifert fibred (which includes
all spherical $3$-manifolds~\cite[p.\,248]{Thurston}) or hyperbolic,
see e.\,g.~\cite[Theorem 1.1.6]{BBBMP}.
Thus, all closed orientable $3$-manifolds are geometric in the sense of~\cite{TuraevLong}.

It remains to construct a topological $h$-cobordism between $N$ and $M$.
Let $\psi: M\times\R\to N\times\R$ be a homeomorphism.
Since $\psi(M\times\{0\})$ is compact, it is contained
in $N\times (-\infty, T)$ for some $T\gg 0$.
Reversing the $\R$-factor in $M\times\R$ if necessary,
we may assume that $N\times\{T\}\subset \psi(M\times (0,+\infty))$.
Set
$$
W=N\times (-\infty, T]\cap \psi(M\times [0,+\infty))\subset N\times\R.
$$
This is a compact topological manifold with boundary
that defines a topological cobordism between
$M\overset{\mathrm{top}}{=}\psi(M\times\{0\})$
and $N=N\times\{T\}$. By the definition of an $h$-cobordism,
we have to check now that the inclusions of the boundary
components into $W$ are homotopy equivalences.

Let $r_M:M\times\R\twoheadrightarrow M\times [0,+\infty)$ and
$r_N:N\times\R\twoheadrightarrow N\times (-\infty, T]$ be the
obvious strong deformation retractions. Then
$$
r_N\circ \psi \circ r_M\circ \psi^{-1}: N\times\R \twoheadrightarrow W
$$
is a strong deformation retraction. Hence, the inclusion
$W\hookrightarrow N\times\R$ is a homotopy equivalence.
Since the inclusions $N\times \{T\}\hookrightarrow N\times\R$
and $\psi(M\times\{0\})\hookrightarrow N\times\R$ are
also homotopy equivalences, it follows that $W$ is
a topological $h$-cobordism indeed,
which completes the proof of Theorem~\ref{mainB}.

\end{document}